\documentclass[12pt]{article}

\usepackage{titlesec}

\usepackage[utf8]{inputenc}

\usepackage{comment}
\usepackage{amsmath}
\usepackage{amssymb}
\usepackage{amsthm}
\usepackage{amsfonts}
\usepackage{graphicx}
\usepackage{dsfont}

\usepackage[noadjust]{cite}

\usepackage{bbm}

\usepackage{comment}
\usepackage{cite}

\usepackage[pdfstartview={XYZ null null 1.00},pdfpagemode=UseNone]{hyperref}

\newtheorem*{theorem*}{Theorem}
\newtheorem*{statement*}{Proposition}
\newtheorem{theorem}{Theorem}[section]
\newtheorem{lemma}{Lemma}[section]
\newtheorem{cor}{Corollary}
\newtheorem{example}{Example}
\newtheorem{rem}{Remark}
\newtheorem*{acks}{Acknowledgments}

\newcommand{\aseq}{\stackrel{\text{a.s.}}{=}}
\newcommand{\asleq}{\stackrel{\text{a.s.}}{\leq}}
\newcommand{\asgeq}{\stackrel{\text{a.s.}}{\geq}}
\newcommand{\pconv}{\xrightarrow{\mathbb{P}}}
\newcommand{\asconv}{\xrightarrow{\text{a.s.}}}

\newcommand{\prob}{\mathbb{P}}
\newcommand{\E}{\mathbb{E}}
\newcommand{\sh}{\sinh}
\newcommand{\ch}{\cosh}

\newcommand{\hup}{H^{(upper)}}
\newcommand{\hlo}{H^{(lower)}}
\newcommand{\hh}{\widetilde{h}}

\title{Asymptotic distribution of the derivative of the taut string accompanying Wiener process}

\author{ M.A. Lifshits, A.A. Podchishchailov}

\date{\today}

\begin{document}
	
\maketitle

\begin{abstract}
    In the article, we find the asymptotic distribution of the derivative of the taut string accompanying a Wiener process $W$ in a strip of fixed width on long time intervals. This enables to find explicit expressions for the limits
\[
  \lim_{T\to\infty} \frac{1}{T} \inf \int_0^T \varphi(h'(t))\, dt,
\]
where infimum is taken over the class of functions $h$ satisfying assumptions
\[
    h\in AC[0,T], \ h(0)=(0),\ h(T)=W(T),\ \max_{t\in [0,T]} |h(t)-W(t)|
    \le \frac {r}{2}. 
\]    
For example, for kinetic energy $\varphi(u)=u^2$ which was considered earlier by
Lifshits and Setterqvist \cite{Lifshits2015}, the limiting value equals to $\tfrac{\pi^2}{6r^2}$.
\end{abstract}

\medskip

%%%%%%%%%%%%%%%%%%%%%%%%%%%%%%%%%%%%%
\section{Introduction and main results}
%%%%%%%%%%%%%%%%%%%%%%%%%%%%%%%%%%%%%

%%%%%%%%%%%%%%%%%%%
\subsection{Taut strings}
%%%%%%%%%%%%%%%%%%%
First of all, let us recall the notion of taut string and the related background.
A function $\varphi: {\mathbb{R}}\mapsto {\mathbb{R}}$ is called {\it convex}, if
\begin{equation} \label{convex}
  \varphi(\alpha\, x+(1-\alpha)\, y) \le \alpha\, \varphi (x)+(1-\alpha)\, \varphi(y),
  \qquad 0\le \alpha \le 1; x,y\in {\mathbb{R}}.
\end{equation}
A function $\varphi$ is called {\it strictly} convex if the inequality in \eqref{convex}
is strict for $0<\alpha<1$, $x\not= y$.

Let $\text{AC}[T_1,T_2]$ denote the class of all absolutely continuous functions on $[T_1,T_2]$.
For a convex function $\varphi$, consider the integral functional
\[
  G^\varphi_{[T_1,T_2]}(h) 
  := \int_{T_1}^{T_2} \varphi(h'(t))\,dt,
  \qquad h\in \text{AC}[T_1,T_2].
\]
In general, $ G^\varphi_{[T_1,T_2]}(h)$ may be interpreted as a kind of "energy" (or a measure of variability) of the function $h$.
The most interesting examples are $\varphi(u)=\sqrt{u^2+1}$ and $\varphi(u)=\vartheta_p(u):=|u|^p$ for $p\geq 1$. In these cases, $ G^\varphi_{[T_1,T_2]}$ is the graph length of $h$, the total variation of $h$ (for $p=1$) and "kinetic energy" of $h$ (for $p=2$), respectively.
Notice, by the way, that in these examples $\varphi$ is strictly convex, except for $\vartheta_1(u)=|u|$.

Let us now introduce an optimization problem to deal with. Let $\hlo(\cdot),\hup(\cdot)$ be two continuous functions on $[T_1,T_2]$ such that
\[
   \hlo(t)<\hup(t), \qquad T_1\le t\le T_2,
\]
and let $\hh_1,\hh_2$ be the real numbers satisfying
\[
   \hlo(T_i)\le \hh_i\le \hup(T_i), \qquad i=1,2.
\]
Consider now an optimization problem
\begin{equation} \label{general_prob} 
     G^\varphi_{[T_1,T_2]} (h) \searrow \min
\end{equation}
under the constraints $h\in \text{AC}[T_1,T_2]$, $h(T_1)=\hh_1, h(T_2)=\hh_2$, and
\[
      \hlo(t)  \le  h(t) \le  \hup(t), \qquad T_1\le t\le T_2.
\]

This means that we are looking for a function spending the minimal amount of energy among all functions taking an initial value $\hh_1$, a terminal value $\hh_2$ and running between the lower boundary $\hlo(\cdot)$ and the upper boundary $\hup(\cdot)$.

The remarkable facts well known in optimization theory are the following (see e.g. \cite[Appendix]{Schertzer2018} essentially based on 
\cite{Grasmair2007}).
\begin{itemize}
\item  There exists a solution of \eqref{general_prob} which is the same for all convex $\varphi$ given $\hh_1$, $\hh_2$, $\hlo$, $\hup$.
This universal solution is called "taut string".
\item  If $\varphi$ is {\it strictly} convex, then the taut string is the unique solution of \eqref{general_prob}.
\end{itemize}

Note that for a convex but non-strictly convex $\varphi$ the solution  of \eqref{general_prob} need not be unique. In particular, for
$\varphi(u)=|u|$ a solution alternative to the taut string and called "lazy function" may appear for certain initial and terminal conditions.
In short, the lazy function is a function that remains constant in time as long as this is allowed by the boundary conditions. Lazy functions
were studied in \cite{Lochowski2008}--\cite{Lochowski2013b}.

%%%%%%%%%%%%%%%%%%%
\subsection{Taut strings accompanying Wiener process}
%%%%%%%%%%%%%%%%%%%
From now on, we will consider a special case of \eqref{general_prob}, where the upper and the lower boundaries follow the trajectories of a Wiener process $W$. Namely, for $T>0$, $r>0$ we let in \eqref{general_prob}
$T_1:=0$, $T_2:=T$, $\hh_1:=0$, $\hh_2:=W_T$, $\hlo(t):=W_t-r/2$, $\hup(t):=W_t+r/2$ and write $G^\varphi_T$ instead of $G^\varphi_{[0,T]}$.

The corresponding taut string is called {\it the taut string accompanying Wiener process $W$ on the time interval $[0,T]$}. We will denote it $\eta_{T,r}$.
Introducing the notation for uniform norms 
\[ 
          \|g\|_{[T_1,T_2]} := \sup_{T_1\le t\le T_2} |g(t))|, \quad   \|g\|_T := \|g\|_{[0,T]},
\]
we have, by the definition of taut string, 

\begin{equation*}
	G^\varphi_T(\eta_{T,r}) 
         = \int_0^T \varphi(\eta'_{T,r}(t))\, dt
         = I^0_W(T,r,\varphi)
\end{equation*}
where
\begin{equation}   \label{eq:I0_def}
	\begin{aligned}
		I^0_W(T,r,\varphi):=\min \{ G^{\varphi}_T(h);h\in\text{AC}[0,T],\|h-W\|_T \leq \frac{r}{2},h(0)=0,&
		\\
		h(T)=W_T\}&,
	\end{aligned}
\end{equation}
for all $T>0$, $r>0$ and all convex functions $\varphi$.

In the work~\cite{Lifshits2015} the asymptotic behavior of $I^0_W$ 
was considered for kinetic energy, i.e., for the function $\varphi(u):=\vartheta_2(u)=u^2$. It was shown that

\begin{equation}
	\label{eq:L_2_unknown}
	\begin{aligned}
		\lim\limits_{T\rightarrow\infty}
		\frac{1}{T}
		\,
		I^0_W(T,r,\vartheta_2)
		=
		\lim\limits_{T\rightarrow\infty}
		\frac{1}{T}\,
		\int_0^T \eta'_{T,r}(t)^2 \, dt
		=\frac{4\mathcal{C}^2}{r^2}, 
	\end{aligned}
\end{equation}
where one has a.s.-convergence and convergence in $L_q$ for every $q>0$. Here
$\mathcal{C}$  is some positive constant unknown to the authors of~\cite{Lifshits2015}; they provided some theoretical bounds for it and, by computer simulation, obtained an approximate value $\mathcal{C}\approx 0.63$. 
A similar result (with the same limit value) is also true for the free-end problem, i.e. for 
\begin{equation}
	\label{eq:I_def}
		I_W(T,r,\varphi):=\min \{G^{\varphi}_T(h);h\in\text{AC}[0,T],
                   \|h-W\|_T \leq \frac{r}{2},h(0)=0\}
\end{equation}
in place of $I^0_W(T,r,\varphi)$. 

Various extensions of \eqref{eq:L_2_unknown}  (for strips of variable width, for random walks, for unilateral constraints) can be found
in~\cite{Blinova2020,Nikitin2024,Siuniaev2021}.

Later on, in~\cite{Schertzer2018} a result analogous to~\eqref{eq:L_2_unknown}, i.e. the existence of the limit
\begin{equation}
	\label{eq:gen_limit_def}
	R_{r}(\varphi)
	:= \lim\limits_{T\rightarrow\infty}
	\frac{1}{T}\, \int_0^T \varphi(\eta'_{T,r}(t))\, dt
\end{equation}
was obtained for a rather broad class of functions $\varphi$ (asserting convergence to an unknown constant depending on $\varphi$), as well as a central limit theorem for  $G^\varphi_T(\eta_{T,r})$. The main assumption here is a power bound on the growth of $\varphi$ at infinity, i.e., one should assume

\begin{equation} \label{eq:power_bound}
		\lim\limits_{|u|\to \infty} \frac{\varphi(u)}{|u|^\alpha}=0
\end{equation}
for some $\alpha>0$.
\medskip
We stress that convexity of $\varphi$ is not assumed in the result but in the non-convex case
the connection to the optimizaton problem is lost and therefore the result becomes much less interesting.

Our initial goal was to find more explicit expressions for the limits \eqref{eq:gen_limit_def}.
This goal is achieved (in particular, the constant from ~\eqref{eq:L_2_unknown} is identified as $\mathcal{C}=\frac{\pi}{\sqrt{24}}$); it became clear that the existence of the limits~\eqref{eq:gen_limit_def} is based on the asymptotic behavior of the distribution (normed sojourn measure) of the derivative of the taut string, i.e. the measure $\nu_{T,r}$ on $\mathbb{R}$ such that

\begin{equation}	\label{eq:nu_def}
	\int\limits_{\mathbb{R}} f(u)\, \nu_{T,r}(du)
	=	\frac{1}{T} \int_0^T f(\eta'_{T,r}(t))\, dt
\end{equation}
for every measurable bounded function $f$. Essentially, for an $r$ fixed, the measures $\nu_{T,r}$ converge weakly, as $T\rightarrow\infty$, to a limiting measure $\nu_{\infty,r}$; the continuous density of the latter admits an explicit computation, namely,

\begin{equation}
	\label{eq:dens_lim_meas}
	p_{\infty,r}(u)
	:= \frac{d\nu_{\infty,r}(u)}{du}
	=
	\begin{cases}
		r\ \frac{ru\coth(ru)-1}{\sh^2(ru)},\ &u\neq 0;\\
		\frac{1}{3}\, r,\ &u=0.
	\end{cases}
\end{equation}
Heuristically, this leads us to equalities

\begin{equation}
\label{eq:R_int_meas}
\begin{aligned}
	R_{r}(\varphi)
	=&
	\lim\limits_{T\rightarrow\infty}
	\frac{1}{T}  \int_0^T \varphi(\eta'_{T,r}(t))\, dt
	=	\lim\limits_{T\rightarrow\infty}
	\int\limits_{\mathbb{R}} \varphi(u)\, \nu_{T,r}(du)
	\\
	=&\int\limits_{\mathbb{R}} \varphi(u)\, \nu_{\infty,r}(du)
	=	\int\limits_{\mathbb{R}} \varphi(u)p_{\infty,r}(u)\, du,
\end{aligned}
\end{equation}
which solves the problem of finding the limit~\eqref{eq:gen_limit_def}; 
convergence of sojourn measures is not related to particular functions $\varphi$. 
For example, substituting  the quadratic function from~\eqref{eq:L_2_unknown}, we find (cf. Section~\ref{sec:t2_calc} below) 

\begin{equation*}
	\begin{aligned}
	4r^{-2}\mathcal{C}^2
	=&
	\int\limits_{\mathbb R} u^2r \frac{(ru\coth(ru)-1)}{\sh^2(ru)} \, du
	= \frac{\pi^2}{6r^2}.
	\end{aligned}
\end{equation*}
Thus,

\begin{equation*}
	\mathcal{C}=\frac{\pi}{\sqrt{24}}\approx 0.641,
\end{equation*}
which agrees well with the simulated value from~\cite{Lifshits2015} mentioned above.
\medskip

Our approach to the rigorous confirmation of~\eqref{eq:R_int_meas} goes through consideration of some special family of functions $\varphi$.
Recall the usual notation
\[
    v^+:= \max\{v,0\}, \quad v^-:= \max\{-v,0\}, \qquad v\in \mathbb{R},
\]
and define the family of functions
\begin{equation} \label{eq:rhomu}
	\rho_\mu^{+}(u):=(u-\mu)^+,\ \mu\in\mathbb{R}.
\end{equation}

For handling these functions, we need the notion and properties of 
truncated variation defined and studied in the works of R. Łochowski, 
cf.~\cite{Lochowski2008, Lochowski2011, Lochowski2013b, Lochowski2013}. 
We will show that the results of~\cite{Lochowski2013} yield

\begin{equation}   \label{eq:rho^+_val_def}
	R_{r}(\rho^{+}_\mu)=
	\int\limits_{\mathbb{R}}   \rho^{+}_\mu(u)p_{\infty,r}(u)\, du
	=
	\begin{cases}
		\frac{1}{2}(\mu \coth{r\mu}-\mu),\ &\mu\neq 0;\\
		\frac{1}{2}r^{-1},\ &\mu=0,
	\end{cases}
\end{equation}
as a special case of~\eqref{eq:R_int_meas}. 
Furthermore, for similar functions $\rho^{-}_{\mu}(u):=(u-\mu)^{-}=\rho^{+}_{-\mu}(-u)$ one has 
$R_{r}(\rho^{-}_{\mu})=R_{r}(\rho^{+}_{-\mu})$ by the symmetry of Wiener process distribution. 
The weak convergence $\nu_{T,r}\Rightarrow\nu_{\infty,r}$, as $T\rightarrow\infty$, also follows from~\eqref{eq:rho^+_val_def}.

Since sufficiently general convex functions can be represented as a mix of the functions
$\rho^{+}_\mu$ and $\rho^{-}_\mu$, by using linearity of the functional $R_{r}(\cdot)$, 
one can obtain~\eqref{eq:R_int_meas} for those functions, too.

Unfortunately, when considering the limits like \eqref{eq:R_int_meas} rigorously, 
one has to distinguish a.s.-convergence and convergence in probability.
Our main result in this direction is as follows.

\begin{theorem}	\label{th:main}
%%For every convex function $\varphi$ it is true that
If a convex function $\varphi$ satisfies
\begin{equation} \label{eq:lambda}
\int\limits_{\mathbb{R}}  \varphi(u) e^{-2\lambda|u|} du <\infty
\quad \textrm{for some } \lambda\in (0,r)
\end{equation}
or
\begin{equation} \label{eq:rinfty}
     \int\limits_{\mathbb{R}}  \varphi(u) \, p_{\infty,r}(u) du =\infty
\end{equation}
then
\begin{equation} \label{main_pconv}
  \frac{1}{T}  \int_0^T \varphi(\eta'_{T,r}(t))\, dt
  \pconv \int\limits_{\mathbb{R}} \varphi(u) \,  p_{\infty,r}(u)\, du  
\end{equation}
and
\begin{equation} \label{main_lower}
   \liminf_{T\to\infty} \frac{1}{T}  \int_0^T \varphi(\eta'_{T,r}(t))\, dt
   \asgeq \int\limits_{\mathbb{R}} \varphi(u) \,  p_{\infty,r}(u)\, du.    
\end{equation}

If a function $\varphi$ is continuous almost everywhere, locally bounded and satisfies power growth condition \eqref{eq:power_bound}, then
\begin{equation} \label{main_asconv}
      \frac{1}{T}  \int_0^T \varphi(\eta'_{T,r}(t))\, dt
      \asconv \int\limits_{\mathbb{R}} \varphi(u) \, p_{\infty,r}(u)\,du.
\end{equation}
\end{theorem}
\medskip

\begin{rem} Actually the gap between \eqref{eq:lambda} and \eqref{eq:rinfty} which is not covered by Theorem \ref{th:main} is very narrow. Since the density $p_{\infty,r}(\cdot)$ is locally bounded and
\begin{equation} \label{eq:pasymp}
   p_{\infty,r}(u) \sim \frac{4r^2|u|}{e^{2 r |u|}}, \qquad \textrm{as } |u|\to\infty, 
\end{equation}
the gap consists of the functions $\varphi$ such that 
\[
    \int\limits_{\mathbb{R}} \varphi(u)\, p_{\infty,r}(u) du <\infty
\]
but 
\[
    \int\limits_{\mathbb{R}} \varphi(u)  p_{\infty,\lambda}(u) du =\infty 
    \qquad \textrm{for all } \lambda\in (0,r).
\]
We believe that \eqref{eq:lambda} is just a technical condition and convergence in probability actually holds for every convex function $\varphi$.

Note also that \eqref{eq:lambda} is much weaker than \eqref{eq:power_bound}.
\end{rem}

\medskip
In general, one can not replace in~\eqref{main_pconv} convergence in probability by a.s.-convergence because of the following result. 

\begin{theorem} \label{th:contrex}
Let $\lambda>1$, $r>0$ and $\varphi(u)\ge e^{\lambda r u}$ for all sufficiently large
$u\in \mathbb{R}$. Then
	\begin{equation*}
		\limsup_{T\to\infty} \frac 1T \int_0^T\varphi(\eta'_{T,r}(t))\, dt \aseq +\infty. 
	\end{equation*}
\end{theorem}

Notice that if $\varphi(u)=e^{\lambda r u}$ with $\lambda\in (1,2)$, then the assumption
of Theorem~\ref{th:contrex} is satisfied and, in view of the density asymptotics of the limiting measure \eqref{eq:pasymp}, it is true that
\[  
   \int_0^T \varphi(u) \, p_{\infty,r}(u)\,du <\infty.
\]
Therefore, there is no a.s. convergence in~\eqref{main_pconv}. 
It would be interesting to identify an exact upper bound for the growth 
of $\varphi$ beyond  which a.s.-convergence may be lost.
\medskip

The article is organized as follows. In Section~\ref{sec:TV} we recall the necessary information on truncated variation; this notion is applied to our problem in
Section~\ref{sec:asymp}. Convergence of sojourn measures and the lower 
bound~\eqref{main_lower} are obtained in Section~\ref{sec:asympt_meas}. 
In Section~\ref{sec:pconv} we prove the main result of the article -- the weak law of large numbers~\eqref{main_pconv}. 
In Section~\ref{sec:SLLN} we prove the strong law of large numbers~\eqref{main_asconv}. 
Finally, Section~\ref{sec:addendum} contains some secondary but necessary proofs including the proof of Theorem~\ref{th:contrex}.
\medskip

%%%%%%%%%%%%%%%%%%%%%%%%%%%%%%%%%%%%%%%%%%%%%%%%%%%%%%%%%%%%%%%%%%%%%%%%%
\section{Truncated variation}
%%%%%%%%%%%%%%%%%%%%%%%%%%%%%%%%%%%%%%%%%%%%%%%%%%%%%%%%%%%%%%%%%%%%%%%%%
\label{sec:TV}

%%%%%%%%%%%%%%%%
\subsection{Basic notions and properties}
%%%%%%%%%%%%%%%%

Let us recall the main concepts and results from the articles~\cite{Lochowski2008,Lochowski2011,Lochowski2013,Lochowski2013b}. 
{ \it Truncated variation} of a c\`adl\`ag function $f$  on an interval $a,b]$ for a given $r\ge 0$ is defined as
\begin{equation*}
	\begin{aligned}
		TV^r(f,[a,b]):=
		\sup_n
		\sup_{a\leq t_1<t_2< \ldots < t_n\leq b}
		\sum\limits_{k=1}^{n-1}
		\rho^{+}_r(|f(t_{k+1})-f(t_{k})|),
	\end{aligned}
\end{equation*}
where expression $\rho^{+}_r(\cdot)$ was defined in \eqref{eq:rhomu}. Notice that $TV^0(f,[a,b])$ is the standard total variation of $f$.

Truncated variation can be expressed as the solution of the following optimization problem
\begin{equation}
	\label{eq:TV_opt}
	TV^{r}(f,[a,b])
	=
	\inf
	\bigg\{
	TV^0(g,[a,b]); \| g-f \|_{[a,b]}\leq \frac{r}{2}
	\bigg\}.
\end{equation}
For continuous $f$, there is a continuous function (lazy function) at which the minimum in~\eqref{eq:TV_opt} is attained.

Moreover, if $f$ is continuous, then we can consider the infimum in~\eqref{eq:TV_opt} only over the class of absolutely continuous functions. Recall that for AC functions the total variation can be represented as
\begin{equation} \label{eq:TV_integral}
	TV^0(g,[a,b])
	=	\int_{a}^b |g'(t)|\, dt,
	  \quad g\in\text{AC}[a,b].
\end{equation}

\begin{lemma} \label{lemma:TVr_AC}
Let $f$ be continuous on an interval $[a,b]$; Then for any $r>0$
	\begin{equation}
            \label{eq:TVAC}
		TV^r(f,[a,b]) = \inf\left\{ \int_a^b |g'(t)| dt ; \  g\in {\normalfont\text{AC}}[a,b], \|g-f\|_{[a,b]}\le \frac r2  \right\}.
	\end{equation}
\end{lemma}

The {\it up\-ward truncated variation}  $UTV^{r}(f,[a,b])$ and the {\it down\-ward truncated variation}  $DTV^{r}(f,[a,b])$ are defined in the same way:
\begin{equation*}
	\begin{aligned}
		UTV^r(f,[a,b]):=
		\sup_n
		\sup_{a\leq t_1<t_2< \ldots < t_n\leq b}
		\sum\limits_{k=1}^{n-1}
		\rho^{+}_r(f(t_{k+1})-f(t_{k})) 
	\end{aligned}
\end{equation*}
and $DTV^{r}(f,[a,b]):=UTV^{r}(-f,[a,b])$.

Similarly to~\eqref{eq:TV_opt}, there is the following representation
\[
%%	\label{eq:UTV_opt}
	UTV^{r}(f,[a,b])
	=
	\inf
	\bigg\{
	UTV^0(g,[a,b]); \| g-f \|_{[a,b]}\leq \frac{r}{2}
	\bigg\}.
\]

Similarly to~\eqref{eq:TV_integral}, there is an integral representation for  AC functions:
\begin{equation} \label{eq:UTV_integral}
	UTV^0(g,[a,b])
	=
	\int_{a}^b
	g'(t)^+\, dt,
	\quad
	g\in\text{AC}[a,b].
\end{equation}

Finally, similarly to~\eqref{eq:TVAC}, there is the following representation
\begin{equation} \label{eq:UTVAC}
		UTV^r(f,[a,b]) = \inf\left\{ \int_a^b g'(t)^+ dt ; \  g\in {\normalfont\text{AC}}[a,b], 
        \|g-f\|_{[a,b]}\le \frac r2  \right\}.
\end{equation}

Since the representations~\eqref{eq:TVAC} and~\eqref{eq:UTVAC} are not considered in the mentioned Łochowski's articles, we will prove them in Section~\ref{subsec:TVAC}.

%%%%%%%%%%%%%%%%%
\subsection{Truncated variation of the Wiener process}
%%%%%%%%%%%%%%%%%
In the article~\cite{Lochowski2013}, the asymptotics of truncated variations of a Wiener process with a drift is found. This result is important for our subsequent calculations.

\begin{theorem}[\cite{Lochowski2013}]
	\label{th:TV_asympt}
	Let $(W_t)_{t\geq 0}$ be a Wiener process; for $\mu\in\mathbb{R}$ define
	\begin{equation*}
		X_t := W_t- \mu t,\quad t\geq 0.
	\end{equation*}
	Then, as $T\to +\infty$,
	\begin{equation*}
		\begin{aligned}
		\frac{1}{T}\,TV^{r}(X,[0,T])
		&\asconv
		m_{r}(\mu),
            \\
            \frac{1}{T}\,UTV^{r}(X,[0,T])
		&\asconv
		q_r(\mu),
		\end{aligned}
	\end{equation*}
	where
	\begin{equation}
		\label{eq:kappa_lim_value}
		m_{r}(\mu)=
		\begin{cases}
			\mu \coth({r\mu}),\ &\mu\neq 0;\\
			r^{-1},\ &\mu=0;
		\end{cases}
        \end{equation}
        \begin{equation}
        \label{eq:rho_lim_value}
		q_r(\mu)=
		\begin{cases} \frac12 (\mu \coth(r\mu)-\mu),& \mu \not=0; \\
		\frac12 r^{-1},& \mu=0.
	\end{cases}
	\end{equation}
\end{theorem}

In the next section we will establish a connection between truncated variation and taut strings. Also, using Theorem~\ref{th:TV_asympt}, we will find the value of $R_r(\rho^{+}_\mu)$.

%%%%%%%%%%%%%%%%%%%%%%%%%%%%%%%%%%%%%%%%%%%%%%%%%%%%%%%%%%%%%%
\section{Application of truncated variation}
%%%%%%%%%%%%%%%%%%%%%%%%%%%%%%%%%%%%%%%%%%%%%%%%%%%%%%%%%%%%%%
\label{sec:asymp}

Our first asymptotic result consists in the computation of the functional $R_r$ 
for two fundamental families of functions:\,
$\rho_\mu^+$  defined in \eqref{eq:rhomu} and $\kappa_\mu(u):=|u-\mu|$.

\begin{theorem}
	\label{th:lambda_shift_conv}
	Let $r> 0$, $\mu\in \mathbb{R}$. Then
	\begin{equation} \label{limkappa}
		R_r(\kappa_\mu) = m_r(\mu),
	\end{equation}
        \begin{equation} \label{limrho}
		R_r(\rho_\mu^+) = q_r(\mu),
	\end{equation}
	where $m_r(\mu)$ and $q_r(\mu)$ are defined in~\eqref{eq:kappa_lim_value} and~\eqref{eq:rho_lim_value}, respectively.
\end{theorem}

\begin{proof}
	First of all, notice that the limits in the definitions of $R_r(\kappa_\mu)$ and $ R_r(\rho_\mu^+)$, exist by Schertzer's theorem~\cite[Theorem 1.2]{Schertzer2018}, since 
 the functions $\kappa_\mu$ and $\rho_\mu^+$ satisfy condition~\eqref{eq:power_bound}.
	
	Consider a Wiener process with a drift $\widetilde{W}_t:=W_t-\mu t$. According to Theorem~\ref{th:TV_asympt}
	\begin{equation} \label{LM}
		\lim_{T\to\infty} \frac{1}{T} \, TV^{r}(\widetilde{W},T) \aseq m_r(\mu).
	\end{equation}
	
By applying Lemma~\ref{lemma:TVr_AC} and making the variable change $g(t)=h(t)-\mu t$, 
we obtain
	\begin{eqnarray*}
		TV^{r}(\widetilde{W},T) &=& \inf\left\{ \int_0^T |g'(t)| dt \ ; \ g\in \text{AC}[0,T], \|g-\widetilde{W}\|_T\le \frac r2  \right\}
		\\
		&=& \inf\left\{ \int_0^T |h'(t)-\mu| dt \ ; \  h\in \text{AC}[0,T], \|h-W\|_T \le \frac r2  \right\}.
	\end{eqnarray*}
	Therefore, for the taut string $\eta_{T,r}$ we have
	\begin{equation*}
   	\int_{0}^{T} \kappa_\mu (\eta_{T,r}'(t)) dt 
         =  \int_{0}^{T} |\eta_{T,r}'(t)-\mu| dt 
         \geq  TV^{r}(\widetilde{W},T).
	\end{equation*}
	From \eqref{LM}, it follows now that $R_r(\kappa_\mu)\ge m_r(\mu)$.
	
In order to get the opposite inequality, let us fix a small number $\delta>0$ 
and find a function $h_T\in \text{AC}[0,T]$, such that
	$\|h_T-W\|_T \le r/2$ and
	\begin{equation*}
	     \int_0^{T} |h_T'(t)-\mu| \, dt \le TV^r(\widetilde{W},T) +\delta.
	\end{equation*}
	Next, we construct an absolutely continuous "pseudostring" $\zeta_{T,r}$ satisfying
    the right border conditions	$\zeta_{T,r}(0)=0$, $\zeta_{T,r}(T)=W_T$
	and such that $\zeta_{T,r}=h_T$ on $[1,T-1]$. The connection parts must be sufficiently smooth,
 satisfy border conditions $\zeta_{T,r}(1)=h_T(1)$, $\zeta_{T,r}(T-1)=h_T(T-1)$,
 %%be differentiable so that
%% \begin{equation*}
%%	  \max_{t\in[0,1]} |\zeta'_{T,r}(t)|<\infty; 
%%       \qquad  
%%       \max_{t\in[T-1,T]} |\zeta'_{T,r}(t)|<\infty;
%%\end{equation*}
and
\begin{equation*}
	   \max \left\{ \|\zeta_{T,r}- W\|_{[0,1]};  
          \|\zeta_{T,r}- W\|_{[T-1,T]} \right\} \le \tfrac{r+\delta}{2}.
\end{equation*}
Such construction provides 
$ \|\zeta_{T,r}- W\|_{T}\le \frac{r+\delta}{2}$ and
$\zeta_{T,r}\in \text{AC}[0,T]$.
	
Let us construct $\zeta_{T,r}$ on $[0,1]$.  
Let
\[
   D := \left\{ w\in C^1[0,1]: w(0)=0, w(1)=W_1, \|w-W\|_{[0,1]} \le \delta/2   \right\}
\]
and
\[
   M := \inf_{w\in D} \max_{0\le t \le 1} |w'(t)|.
\]
Then $D$ is a.s. non-empty, hence, $M<\infty$ a.s.

Let us choose a (random) function $w\in D$ such that
\begin{equation} \label{eq:M1}
     \max_{0\le t\le 1} |w'(t)| \le M+1.
\end{equation}
Since  $w\in D$, we have $w(0)=0, w(1)=W_1$ and $\|w-W\|_{[0,1]}\le \tfrac\delta 2$. Let
\[	%% \label{eq:zeta_0_1_repr}
	    \zeta_{T,r}(t) := w(t) +  (h_T(1)-W_1)t, \qquad t\in[0,1].
\]
Since $|h_T(1)-W_1|\le \frac r2$, it is easy to check that this function has all required properties.
In particular, using \eqref{eq:M1}, $|h_T(1)-W_1|\le \frac r2$, and the definition of $\kappa_\mu$, one has
\[
   \frac{1}{T}\, \int_0^1 \kappa_\mu( \zeta'_{T,r}(t)) \, dt 
   \le \frac{M+1+r/2+\mu}{T} \to 0,
   \quad \textrm{as } T\to \infty,
\]
almost surely, hence, in probability.

By stationarity of the increments of a Wiener process, the construction on $[T-1,T]$ is completely similar (instead of $W$, one should approximate another Wiener process, namely,
$W_T(s):= W_{T-s}-W_T$) and the result has the same properties, in particular,
\[
   \frac{1}{T}\, \int_{T-1}^T \kappa_\mu( \zeta'_{T,r}(t)) \, dt 
   \to 0,
   \quad \textrm{as } T\to \infty,
\]
in probability (yet we do not claim the a.s. convergence).

\medskip
	
	A "pseudostring" beeing constructed, let us start our estimations. By optimality property of the true taut string we have
\[
 %%\label{eq:pseudo_ineq}
		\begin{aligned}
			\int_{0}^{T} \kappa_\mu (\eta'_{T,r+\delta}(t)) \, dt 
            &\le  \int_{0}^{T} \kappa_\mu (\zeta_{T,r}'(t)) \, dt
			\\
			&= \int_{[0,1]\cup[T-1,T]} |\zeta'_{T,r}(t)-\mu| \, dt 
               +  \int_{1}^{T-1} |h_T'(t)-\mu| \, dt
			\\
			&\le \int_{[0,1]\cup[T-1,T]} |\zeta'_{T,r}(t)-\mu| \, dt 
               +  \int_{0}^{T} |h_T'(t)-\mu| \, dt
			\\
			&\le \int_{[0,1]\cup[T-1,T]} |\zeta_{T,r}'(t)-\mu| \,dt 
               +  TV^r(\widetilde{W},T) +\delta.
		\end{aligned}
\]
Divide by $T$ and write the obtained inequality as
\begin{equation*}
	    \frac{1}{T}  \int_{0}^{T} \kappa_\mu (\eta'_{T,r+\delta}(t)) \, dt
        - \frac{1}{T}  TV^r(\widetilde{W},T) - \frac{\delta}{T}
	  \le  \frac{1}{T}  \int_{[0,1]\cup[T-1,T]} |\zeta'_{T,r}(t)-\mu| \, dt.
\end{equation*}
	
	As $T\to\infty$, the left hand side  converges a.s. to the constant
    $R_{r+\delta}(\kappa_\mu)-m_r(\mu)$, while the right hand side converges to
    zero in probability. We infer that
	\begin{equation*}
   	R_{r+\delta}(\kappa_\mu) \le m_r(\mu).
	\end{equation*}
	By switching notation from $r+\delta$ to $r$, we have
	\begin{equation*}
  	R_{r}(\kappa_\mu) \le m_{r-\delta}(\mu).
	\end{equation*}
	Finally, letting $\delta\searrow 0$ and using continuity of the expression $m_{r}(\mu)$ with respect to $r$, we arrive at the required estimate $R_r(\kappa_\mu)\le m_r(\mu)$.
	Therefore, \eqref{limkappa} is proved.
	
	In order to derive \eqref{limrho}, let us consider identity function $\chi(t):=t$.
	By the properties of the Wiener process, we have
	\begin{eqnarray} \nonumber
	R_r(\chi)&=& \lim_{T\to \infty} \frac{1}{T}\int_{0}^{T} \eta'_{T,r}(t)dt
	\\  \label{eq:chi}
    &=&  \lim_{T\to \infty} \frac{ \eta_{T,r}(T)- \eta_{T,r}(0)}{T}
	=  \lim_{T\to \infty} \frac{W_T}{T}=0.
	\end{eqnarray}
	Therefore,
	\begin{equation*}
	R_r(\rho^+_\mu) = R_r\left( \frac 12(\kappa_\mu+\chi-\mu)\right) =  \frac 12( R_r(\kappa_\mu) -\mu)
	= \frac 12 (m_r(\mu)-\mu)=q_r(\mu).
	\end{equation*}
\end{proof}

\begin{rem} For the function $\rho^-_{\mu}(u):= (u-\mu)^-$
	we have
	\begin{equation*}
	\begin{aligned}
	R_r(\rho^-_\mu)&=   R_r\left( \frac 12(\kappa_\mu-(\chi-\mu) )\right) =  \frac 12( R_r(\kappa_\mu) +\mu)
	=  \frac 12 \left(m_r(\mu)+\mu\right) 
   \\
	&=  q_r(-\mu).
		\end{aligned}
	\end{equation*}
\end{rem}

%%%%%%%%%%%%%%%%%%%%%%%%%%%%%%%%%%%%%
\section{Asymptotic distribution of taut string's derivative}
%%%%%%%%%%%%%%%%%%%%%%%%%%%%%%%%%%%%
\label{sec:asympt_meas}

Consider normalized sojourn measures $\nu_{T,r}$ for the derivative of the taut string defined by~\eqref{eq:nu_def}. The following result is true.

\begin{theorem}  \label{th:taut_string_weak_meas}
 As $T\rightarrow\infty$, the normalized sojourn measures $\nu_{T,r}$ a.s.
 converge  weakly to the probability measure $\nu_{\infty,r}$ having the density $p_{\infty,r}$
 given in~\eqref{eq:dens_lim_meas}. Moreover, for every $\mu\in\mathbb{R}$ it is true that
	\begin{equation*}
		R_{r}(\rho^{+}_\mu)=
		\int\limits_{\mathbb{R}} \rho^{+}_\mu(u)p_{\infty,r}(u)\, du.
	\end{equation*}
\end{theorem}
\begin{proof}
	
For all $v\in\mathbb{R}$, $\delta>0$ we have inequalities
	\begin{equation*}
		\frac{1}{\delta}
		(\rho^+_v(u)- \rho^+_{v+\delta}(u) )
		\leq
		\mathds{1}_{u\geq v}
		\leq
		\frac{1}{\delta}
		( \rho^+_{v-\delta}(u)- \rho^+_{v}(u) ).
	\end{equation*}
By integrating the left hand side of the inequality over the measure $\nu_{T,r}$, 
we obtain

\begin{equation*}
\begin{aligned}
	\frac{1}{\delta}
	\int\limits_{\mathbb{R}} (\rho^+_v(u)-\rho^+_{v+\delta}(u)) \, \nu_{T,r}(du)
	=
	\frac{1}{T\delta}
	\int\limits_{0}^{T} 	( \rho^+_v-\rho^+_{v+\delta})(\eta'_{T,r}(t)) \,dt.
\end{aligned}
\end{equation*}
Let us compute the limit, as $T\rightarrow\infty$, of the right hand side of this equality:

\begin{equation*}
	\lim\limits_{T\rightarrow\infty}
	\frac{1}{\delta}
	\int\limits_{\mathbb{R}}( \rho^+_v(u)-\rho^+_{v+\delta}(u))\, \nu_{T,r}(du)
	\aseq 
	\frac{1}{\delta}
	\big(R_r(\rho^+_v) -R_r(\rho^+_{v+\delta}) \big).
\end{equation*}
Therefore,

\begin{equation*}
	\frac{1}{\delta}
	\big(R_r(\rho^+_v)-R_r(\rho^+_{v+\delta}) \big)
	\asleq
	\liminf\limits_{T\rightarrow\infty} 
    \int\limits_{\mathbb{R}} \mathds{1}_{u\geq v}\, \nu_{T,r}(du)
	=	\liminf\limits_{T\rightarrow\infty} \nu_{T,r}[v,\infty).
\end{equation*}

In the same way, we obtain

\begin{equation*}
	\limsup\limits_{T\rightarrow\infty} \nu_{T,r}[v,\infty)
	\asleq
	\frac{1}{\delta}
	\big(R_r(\rho^+_{v-\delta})-R_r(\rho^+_{v}) \big).
\end{equation*}
By letting $\delta\rightarrow0$, we see that the limiting values of both the upper and
lower bounds coincide with $-\frac{d}{dv}R_r(\rho^+_{v})$. Therefore,
\begin{equation}
	\label{eq:dens_int}
	\lim\limits_{T\rightarrow\infty} \nu_{T,r}[v,\infty)
	\aseq	-\frac{d}{dv}R_r(\rho^+_{v})
	= \int\limits_{v}^\infty \frac{d^2}{du^2}R_r(\rho^+_{u})\, du
	:= \int\limits_{v}^\infty p_{\infty,r}(u)\, du.
\end{equation}
This relation shows convergence 
\begin{equation}
	\label{eq:nu_conv_poinwise}
	\lim\limits_{T\rightarrow\infty}
	\nu_{T,r}[v,\infty)
	\aseq	\nu_{\infty,r}[v,\infty)
\end{equation}
for every fixed $v$. Therefore, with probability one~\eqref{eq:nu_conv_poinwise} holds simultaneously for all rational $v$. Furthermore since the distribution tail $\nu_{\infty,r}[v,\infty)$ is continuous in $v$, the convergence holds a.s. simultaneously for all $v\in\mathbb{R}$, which proves weak convergence.

Formulas~\eqref{limrho} and~\eqref{eq:dens_int} provide the asserted expression for $p_{\infty,r}$. Finally, for every $\mu\in\mathbb{R}$ integration by parts yields

\begin{equation*}
	\int\limits_{\mathbb{R}} \rho_\mu^{+}(u)p_{\infty,r}(u)\, du
	=	- \int\limits_{\mu}^\infty \frac{d}{du}R_r(\rho^+_{u})\, du
	=	R_r(\rho_\mu^{+}).
\end{equation*}
\end{proof}

\begin{cor} Let $\varphi$ be either an almost everywhere continuous non-negative or a convex function. 
Then for every $r>0$
  \begin{equation}
  \label{eq:liminf_geq_J}
     \liminf_{T\to\infty} \frac{1}{T}  \int_0^T \varphi(\eta'_{T,r}(t))\, dt
     \asgeq \int\limits_{\mathbb{R}} \varphi(u)  p_{\infty,r}(u)\, du .   
  \end{equation}
\end{cor}

\begin{proof}
 %% \cite[Лемма 2.2]{Vaart2007}), 
If  $\varphi$ is almost everywhere continuous and {\it bounded}, by the well known property
of weak convergence,
\begin{eqnarray*}    
  \lim_{T\to\infty} \frac{1}{T}  \int_0^T \varphi(\eta'_{T,r}(t))\, dt
   =
    \lim_{T\to\infty}   \int\limits_{\mathbb{R}} \varphi(u)\, \nu_{T,r}(du)
   \aseq \int\limits_{\mathbb{R}} \varphi(u) p_{\infty,r}(u)\, du .
\end{eqnarray*} 
Dropping the boundedness assumption, consider for $M>0$ the truncated function
$\varphi_M(u):=\min\{\varphi(u),M\}$. Since $\varphi$ is assumed to be non-negative, $\varphi_M$
is bounded and we may apply the previous chain to it.
Therefore,
\[
   \liminf_{T\to\infty} \frac{1}{T}  \int_0^T \varphi(\eta'_{T,r}(t))\, dt
   \geq
   \lim_{T\to\infty} \frac{1}{T}  \int_0^T \varphi_M(\eta'_{T,r}(t))\, dt
   =  \int\limits_{\mathbb{R}} \varphi_M(u) p_{\infty,r}(u)\, du. 
\]
Furthermore, since 
\[
 \lim_{M\to\infty}  \int\limits_{\mathbb{R}} \varphi_M(u) p_{\infty,r}(u)\, du 
 =  \int\limits_{\mathbb{R}} \varphi(u) p_{\infty,r}(u)\, du,
\]
we obtain \eqref{eq:liminf_geq_J}

On the other hand, a convex function can be represented as a sum of a linear function and
a non-negative continuous function. For linear functions we have a.s. convergence, 
cf.~\eqref{eq:chi}, thus the assertion follows.
\end{proof}

We stress that this corollary is still true in the case when the limiting integral is infinite.

%%%%%%%%%%%%%%%%%%%%%%%%%%%%%%%%%%%%%%%%%%
\section{Weak law of large numbers}
%%%%%%%%%%%%%%%%%%%%%%%%%%%%%%%%%%%%%%%%%%
\label{sec:pconv}
In this section, we prove the weak law of large numbers for the energy of the taut string accompanying Wiener process. It was already stated  in~\eqref{main_pconv} as a part of Theorem~\ref{th:main}. For reading convenience, we repeat it here as a separate theorem.

\begin{theorem} \label{th:weakLLN}
  Let $\varphi$ be a convex non-negative function satisfying \eqref{eq:lambda} or \eqref{eq:rinfty}. Let $r>0$ and let $\eta_{T,r}$ be the taut string accompanying Wiener process in the strip of width $r$ on $[0,T]$; then
\begin{equation*}
		\frac{1}{T}
	  \int\limits_{0}^T \varphi(\eta'_{T,r}(t))\, dt 
		\pconv
		\int\limits_{\mathbb R}
		\varphi(u) p_{\infty,r}(u)\, du,
    \qquad  \textrm{as } T\rightarrow\infty.
	\end{equation*}
\end{theorem}

\begin{proof}
If \eqref{eq:rinfty} holds, i.e. the limiting integral is infinite, then the result follows from~\eqref{eq:liminf_geq_J}. In the sequel, we assume that the limiting integral is finite.

Without loss of generality one may assume that $\varphi(u)=0$ for $u\le 0$. Indeed, any convex function $\varphi$ can be represented as a sum of a linear function $\ell(u):=\varphi(0)+cu$ for $c\in [\varphi_-'(0), \varphi_+'(0)]$ and two convex functions $\varphi_1$ and $\varphi_2$ such that $\varphi_1(u)=0$ for $u\leq 0$ and $\varphi_2(u)=0$ for $u\geq 0$. By ~\eqref{eq:chi}, we have $R_r(\ell)\aseq \varphi(0)$. Therefore, it is enough to prove the theorem for the functions $\varphi_1$ and $\varphi_2$ separately.

From now on, we suppose that $\varphi(u)=0$ for $u\le 0$. In this case, the function $\varphi$ may be represented as
\begin{equation} \label{eq:expansion}
	\varphi(u)=
	\int\limits_{0}^\infty   \rho^{+}_\mu(u)\, \varphi''(d\mu),
\end{equation}
where $\varphi''$ is a locally finite measure. %% in the sense of distributions.
 \medskip

Let $\theta: [0,\infty) \mapsto (0,\infty)$   be a Borel-measurable map such that the constant
$\Theta$ defined as
\[
   \Theta:= \int_0^\infty \theta(\mu)\varphi''(d\mu)
\]
is finite.

As in the proof of Theorem~\ref{th:lambda_shift_conv}, for each $\mu\ge 0$, using \eqref{eq:UTVAC}, we find an absolutely continuous function $h_{T,r,\mu}$ such that
$||h_{T,r,\mu}-W||_T\le \tfrac{r}{2}$ and
	\begin{equation*}
		\int\limits_0^T \rho^{+}_\mu(h'_{T,r,\mu}(t))\,dt
		\leq
		UTV^r(\widetilde{W}_\mu,[0,T]) + \theta(\mu),
	\end{equation*}
where $\widetilde{W}_\mu(t)=W_t-\mu t$ is a Wiener process with a drift.

For fixed $T>2,\delta>0$ we construct a  "pseudostring" $\zeta_{T,r,\mu}$  in the strip of width $r+\delta$ . As in the proof of Theorem~\ref{th:lambda_shift_conv}, let $\zeta_{T,r,\mu}(t)=h_{T,r,\mu}(t)$ on $[1,T-1]$ and for $t\in[0,1]$
\[	\label{eq:zeta_0_1_repr2}
	\zeta_{T,r,\mu}(t) := w(t) +  (h_{T,r,\mu}(1)-W_1)t
\]
where $w$ is the same continuously differentiable function such that $w(0)=0, w(1)=W_1$,
$\|w-W\|_{[0,1]}\le \tfrac\delta 2$ and \eqref{eq:M1} holds. Hence
\begin{equation} \label{eq:maxder}
  \max_{t\in [0,1]} |\zeta'_{T,r,\mu}(t)| \le  \max_{t\in [0,1]} |w'(t)| +  |h_{T,r,\mu}(1)-W_1|  
  \le M+1+  \frac{r}{2}:= \widetilde{M}. 
\end{equation}
This estimate is random but uniform in $\mu,T$.

By stationarity of the increments of a Wiener process, the construction on $[T-1,T]$ is completely similar.

From optimality of the taut string $\eta_{T,r+\delta}$, it follows that
\begin{eqnarray*} 
	\int\limits_0^T
	\rho^{+}_\mu(\eta_{T,r+\delta}'(t))\, dt
 &\le&
 \int\limits_0^T
	\rho^{+}_\mu(\zeta_{T,r,\mu}'(t))\, dt
  \\  
  &=&
		\int\limits_{[0,1]\cup[T-1,T]}
		\rho^{+}_\mu(\zeta_{T,r,\mu}'(t))\, dt
		+
		\int\limits_1^{T-1}
		\rho^{+}_\mu(h_{T,r,\mu}'(t))\, dt
\end{eqnarray*}
for any $\mu$. Using representation~\eqref{eq:expansion}, we get
\begin{equation*}
	\frac{1}{T}
	\int\limits_0^T \varphi(\eta_{T,r+\delta}'(t))\,dt
	\leq
	A^1_{T,r} + B_{T,r}+ A^T_{T,r},
\end{equation*}
where
\begin{equation*}
	\begin{aligned}
		A^1_{T,r}
		&:=
		\frac{1}{T}
        \int\limits_{0}^\infty
		\int\limits_0^1
		\rho^{+}_\mu(\zeta_{T,r,\mu}'(t))\,dt\, \varphi''(d\mu),
		\\
		A^T_{T,r}
		&:=
		\frac{1}{T}\int\limits_{0}^\infty
		\int\limits_{T-1}^T
		\rho^{+}_\mu(\zeta_{T,r,\mu}'(t))\,dt\, \varphi''(d\mu),
		\\
		B_{T,r}
		&:=
			\frac{1}{T}
		\int\limits_{0}^\infty
		\int\limits_1^{T-1}
		\rho^{+}_\mu(h_{T,r,\mu}'(t))\,dt\,\varphi''(d\mu).
\end{aligned}
\end{equation*}
\medskip

Let us estimate the expectation of $B_{T,r}$. Since $\rho^+_\mu$ 
is non-negative, by the definition of $h_{T,r,\mu}$ we have the following inequality

\begin{equation*}
	\frac{1}{T}\int\limits_1^{T-1} \rho^{+}_\mu(h'_{T,r,\mu}(t))\,\,dt
	\leq
	\frac{1}{T}\int\limits_0^T \rho^{+}_\mu(h'_{T,r,\mu}(t))\,\,dt 
	\leq
	\frac{1}{T}  \, UTV^r(\widetilde{W}_\mu,[0,T]) +\frac{\theta(\mu)}{T}.
\end{equation*}

Since $UTV^r(\widetilde{W}_\mu,[0,T])$ is superadditive in $T$ (see~\cite{Lochowski2013}),  it follows from Theorem \ref{th:TV_asympt} and \cite[Lemma 3.6.II]{Daley2005} 
that for every $T>0$
\begin{equation*}
    \frac{1}{T} \, UTV^r(\widetilde{W}_\mu,[0,T])\asleq q_r(\mu). 
\end{equation*}
Fubini's theorem and double integration by parts yield
\begin{equation*}
	\E B_{T,r}
	\leq  \int\limits_{0}^\infty q_r(\mu)\, \varphi''(d\mu) + \frac{\Theta}{T}
	=	\int\limits_{0}^\infty \varphi(u) q_r''(u)\, du     + \frac{\Theta}{T}
	=	\int\limits_{0}^\infty \varphi(u) p_{\infty,r}(u)\, du  +\frac{\Theta}{T}.
\end{equation*} 
\medskip

Let us now estimate the remainder terms $A^1_{T,r}$ and $A^T_{T,r}$.
From~\eqref{eq:maxder}  and the definition of $\rho^{+}_\mu$, it follows that for every $t\in[0,1]$
\[
    0\le \rho^{+}_\mu(\zeta_{T,r,\mu}'(t)) \le 
    \begin{cases}
     \widetilde{M}, & 0\le \mu\le \widetilde{M},
     \\ 
     0,&  \mu>\widetilde{M}. 
    \end{cases}
\]
Therefore,
\[
  0 \le  \int\limits_{0}^\infty \int\limits_0^1
		\rho^{+}_\mu(\zeta_{T,r,\mu}'(t))\,dt\, \varphi''(d\mu)
     \le \widetilde{M} \varphi''[0,\widetilde{M}],
\]
whence
\[
    0 \le A^1_{T,r} \le \frac{ \widetilde{M} \varphi''[0,\widetilde{M}]}{T} \pconv 0, 
    \quad \textrm{ as } T\to\infty.
\]

%%	\label{eq:A_major}

Since $A^1_{T,r}$ and $A^T_{T,r}$ have the same distributions (by stationarity of the increments of a Wiener process), we also obtain $ A^T_{T,r}\pconv 0$. 

Recall that the limiting measure $\nu_{\infty,r}$ depends on $r$ and denote the limiting integral as a function of the strip width:
\begin{equation*}
	J(\lambda) := \int\limits_{0}^\infty \varphi(u) \, \nu_{\infty,\lambda}(du)
                   =  \int\limits_{0}^\infty \varphi(u) \, p_{\infty,\lambda}(u)\, du.
\end{equation*}
Let us show that $J(\cdot)$ is non-increasing; moreover, under assumption \eqref{eq:lambda} $J(\cdot)$ is finite and continuous on some neighborhood of $r$.

Indeed, by the definition of $p_{\infty,\lambda}$ after linear variable change $u=v/\lambda$ one has
\begin{equation} \label{eq:J1}
      J(\lambda)=\int_0^\infty \varphi(v/\lambda) \, p_{\infty,1} (v) \, dv.
\end{equation}
Recall that $\varphi$ is now a convex function vanishing at $(-\infty,0]$. Hence, it is non-decreasing.
It follows from \eqref{eq:J1} that $J(\cdot)$ is non-increasing.

Furthermore, from \eqref{eq:pasymp} and \eqref{eq:J1} it follows that $J(\cdot)$ is finite on some neighborhood of $r$. Finally, if $J(\lambda)<\infty$ for some $\lambda\in(0,r)$, then $J(\cdot)$ is continuous on $(\lambda,\infty)$  by continuity of $\varphi$ and by the Lebesgue dominated convergence theorem --
with $\varphi(\cdot/\lambda)\, p_{\infty,1}(\cdot)$ as an integrable, dominating function in \eqref{eq:J1}. 
This is the point where additional assumption \eqref{eq:lambda} crucially comes into play because
in the sequel we will use that $J(\cdot)$ is left-continuous at $r$.

Let us define
\begin{equation*}
	\begin{aligned}
	S_{T,r}
	&:=
	\frac{1}{T}
	\int\limits_0^T   \varphi(\eta'_{T,r}(t))\,dt,
	\\
	A_{T,r} &:= A^1_{T,r}+A^T_{T,r}	
	\end{aligned}
\end{equation*}
and summarize the results obtained above as follows
\begin{equation*}
	\begin{aligned}
		&1) \, S_{T,r+\delta}\leq A_{T,r} + B_{T,r};
		\\
		&2) \, A_{T,r}\pconv 0;
		\\
		&3) \, \E B_{T,r}\leq J(r)  + \frac{\Theta}{T};
		\\
		&4) \, \liminf_{T\rightarrow\infty} S_{T,r+\delta}\asgeq J(r+\delta),
	\end{aligned}
\end{equation*}
where the last inequality follows from~\eqref{eq:liminf_geq_J}. 
If we replace $r+\delta$ by $r$, we obtain
\begin{equation*}
	\begin{aligned}
		&1) \, S_{T,r}\leq A_{T,r-\delta} + B_{T,r-\delta};
		\\
		&2) \, A_{T,r-\delta}\pconv 0;
		\\
		&3)\, \E B_{T,r-\delta}\leq J(r-\delta)+ \frac{\Theta}{T};
		\\
		&4) \, \liminf_{T\rightarrow\infty} S_{T,r}\asgeq J(r).
	\end{aligned}
\end{equation*}
We will prove now that $S_{T,r}\pconv J(r)$ using only properties $1)$ -- $4)$.

For each $y < J(r)$  we have
\begin{eqnarray*}
   \prob(B_{T,r-\delta}< y)
   &\leq&  \prob(S_{T,r}-A_{T,r-\delta}< y)
\\  
  &\leq&  \prob\left(S_{T,r}-A_{T,r-\delta}< y, S_{T,r} \ge \frac{J(r)+y}{2}\right) 
          +  \prob\bigg(S_{T,r}<\frac{J(r)+y}{2}\bigg)
\\  
   &=& \prob\bigg( A_{T,r-\delta}> S_{T,r} - y > \frac{J(r)+y}{2} - y =\frac{J(r)-y}{2}   \bigg) 
      +  \prob\bigg(S_{T,r}<\frac{J(r)+y}{2}\bigg)
\\	
   &\leq& \prob\bigg(A_{T,r-\delta}>\frac{J(r)-y}{2}\bigg)
            + \prob\bigg(S_{T,r}<\frac{J(r)+y}{2}\bigg).
\end{eqnarray*}
From properties $2)$ and $4)$ it follows that both terms tend to zero as $T\to\infty$. Hence,
\begin{equation*}
	\begin{aligned}
	\E(B_{T,r-\delta}-J(r-\delta))^-
	&\leq
	J(r-\delta)\prob(B_{T,r-\delta}<y) + (J(r-\delta)-y)\cdot 1
	\\
	&=
	o(1)  + (J(r-\delta)-y).
\end{aligned}
\end{equation*}
By letting $y\nearrow J(r)$, we obtain
\begin{equation*}
   \limsup_{T\to\infty}	\E(B_{T,r-\delta}-J(r-\delta))^-
	\leq   J(r-\delta)-J(r).
\end{equation*}
Now, using equality
\begin{equation*}
	\E(B_{T,r-\delta}-J(r-\delta))^+
	=  \E(B_{T,r-\delta}-J(r-\delta))
	+  \E(B_{T,r-\delta}-J(r-\delta))^-
\end{equation*}
and property $3)$, we get
\begin{equation}  \label{eq:B+_limsup}
	\limsup_{T\to\infty}  \E(B_{T,r-\delta}-J(r-\delta))^+
	\leq J(r-\delta)-J(r).
\end{equation}

Let us fix two small parameters $\varepsilon>\varepsilon_1>0$ and 
choose $\delta>0$ such that $J(r-\delta)-J(r)<\varepsilon_1$. 
Using Markov's inequality, we obtain
\begin{equation*}
	\begin{aligned}
		\prob(B_{T,r-\delta}>J(r)+\varepsilon)
		&=
		\prob(B_{T,r-\delta}>(J(r)-J(r-\delta))+J(r-\delta)+\varepsilon)
		\\
		&\leq
		\prob(B_{T,r-\delta}>-\varepsilon_1+J(r-\delta)+\varepsilon)
		\\
		&=
		\prob(B_{T,r-\delta}-J(r-\delta)>\varepsilon-\varepsilon_1)
		\\
		&\leq
		\frac{\E(B_{T,r-\delta}-J(r-\delta))^+}{\varepsilon-\varepsilon_1}.
	\end{aligned}
\end{equation*}
Now, by~\eqref{eq:B+_limsup},
\begin{eqnarray*}
		\limsup_{T\to\infty}   \prob(B_{T,r-\delta}> J(r)+ \varepsilon)
		&\leq&
		\limsup_{T\to\infty}
		\frac{\E(B_{T,r-\delta}-J(r-\delta))^+}{\varepsilon-\varepsilon_1}
\\
        &\leq&
		\frac{J(r-\delta)-J(r)}{\varepsilon-\varepsilon_1}
        \leq \frac{\varepsilon_1}{\varepsilon-\varepsilon_1}.
\end{eqnarray*}
From properties $1)$ and $2)$, it follows that
\[
  \limsup_{T\to\infty}   \prob(S_{T,r}> J(r)+ 2\varepsilon) 
  \leq  \frac{\varepsilon_1}{\varepsilon-\varepsilon_1}.
\]
By letting $\varepsilon_1\to 0$, we obtain
\begin{equation*}
	\lim_{T\to\infty}  \prob(S_{T,r}> J(r)+ 2\varepsilon) =0.
\end{equation*}
This fact and property $4)$ prove that
\begin{equation*}
	S_{T,r} \pconv J(r).
\end{equation*}
\end{proof}

%%%%%%%%%%%%%%%%%%%%%%%%%%%%%%%%%%%%%%%%%%
\section{Strong law of large numbers}
%%%%%%%%%%%%%%%%%%%%%%%%%%%%%%%%%%%%%%%%%%
\label{sec:SLLN}

In this section we prove the following result which is a part of the main 
Theorem~\ref{th:main}, cf. \eqref{main_asconv}.

\begin{theorem}
	\label{th:strong_LLN}
	Let a function $\varphi$ be continuous almost everywhere, locally bounded and satisfy the condition of power growth \eqref{eq:power_bound}. 
Then for every $r>0$ it is true that 
  \begin{equation*}
			R_{r}(\varphi)
			\aseq
			\int\limits_{\mathbb R} \varphi(u) p_{\infty,r}(u)\,du.
	\end{equation*}
\end{theorem}

\begin{proof} 
First of all, let us note that that the limit $R_{r}(\varphi)$ exists a.s. by Schertzer's result. We only need to identify it with the integral from Theorem's assertion.

Without loss of generality, we may assume that $\varphi\ge 0$.

Furthermore, let us note that in two cases the theorem's assertion follows immediately from our previous results.

\begin{itemize}
\item If  $\varphi$ is almost everywhere continuous  and {\it bounded}, then the asserted form of the limit follows from the a.s. convergence of the sojourn measures 
(Theorem~\ref{th:taut_string_weak_meas});
\item If $\varphi$ is {\it convex}, then we have convergence in probability to the required limit by Theorem~\ref{th:weakLLN}, hence the a.s.-limit $R_{r}(\varphi)$ is the same (note that the power growth condition yields \eqref{eq:lambda}, hence Theorem~\ref{th:weakLLN} may be applied) .
\end{itemize}

Let us now consider the general case. Let $M>0$. We represent $\varphi$ as
$\varphi=\varphi_{1,M} + \varphi_{2,M}$, where
\begin{eqnarray*}
    \varphi_{1,M}(u) &=& \min\{\varphi(u),M\};
\\    
     \varphi_{2,M}(u) &=& (\varphi(u)-M)^+. 
\end{eqnarray*}
Since the function $\varphi_{1,M}$ is almost everywhere continuous and bounded, 
we have
\[
  R_r\left( \varphi_{1,M} \right) 
 = \int\limits_{\mathbb R} \varphi_{1,M}(u) p_{\infty,r}(u)\,du 
 \leq \int\limits_{\mathbb R} \varphi(u) p_{\infty,r}(u)\,du.
\]
By the power growth condition, for some  $c_1,c_2>0,\alpha\ge 1$  for all 
$u\in\mathbb{R}$ it is true that $\varphi(u)\le c_1+c_2 |u|^\alpha:= L(u)$. 
Therefore,
\[
   \varphi_{2,M}(u)  \le (L(u)-M)^+,
\]
where the right hand side is a convex function. Hence,
\[
    R_r\left( \varphi_{2,M} \right) \leq  R_r\left( (L(\cdot)-M)^+ \right)
    = \int\limits_{\mathbb R} (L(u)-M)^+ p_{\infty,r}(u)\,du.
\]
Thus,
\begin{eqnarray*}
  R_r(\varphi) =   R_r\left( \varphi_{1,M} \right) +  R_r\left( \varphi_{2,M} \right) 
  \le \int\limits_{\mathbb R} \varphi(u) p_{\infty,r}(u)\,du
    +  \int\limits_{\mathbb R} (L(u)-M)^+ p_{\infty,r}(u)\,du.
\end{eqnarray*}
Take a limit in $M\nearrow+\infty$  and notice that the latter integral tends to zero.
Therefore,
\[
   R_r(\varphi) \le \int\limits_{\mathbb R} \varphi(u) p_{\infty,r}(u)\,du.
\]
The opposite estimate follows from \eqref{eq:liminf_geq_J}.
\end{proof}

%%%%%%%%%%%%%%%%%%%%%%%%%%%%%%%%%%%%%%%%%%%%%%%%%%%%%%%%%%%%
\section{Various proofs}
\label{sec:addendum}

%%%%%%%%%%%%%%%%%%%%%%%%%%%%%%%%%%%%%%%%%%%%%%%%%%%%%%%%%%%%
\subsection{Proof of Theorem \ref{th:contrex}}
%%%%%%%%%%%%%%%%%%%%%%%%%%%%%%%%%%%%%%%%%%%%%%%%%%%%%%%%%%%%	
	Let us fix $r>0$ and denote $a:=\frac{r^2}{2}$. Let $T_n:=n$ and consider the events
	\begin{equation*}
	    A_n:= \left\{ W(T_n)- W\left(T_n-\tfrac{a}{\ln T_n}\right) \ge r \right\}.
	\end{equation*} 
	Since for large $n$ the intervals $[T_n-\tfrac{a}{\ln T_n}, T_n]$ do not overlap, the corresponding events $A_n$ are independent.
	We also have
	\begin{equation*}
	\prob(A_n) =\prob\left( {\mathcal N} \ge \frac{r}{(a/\ln T_n)^{1/2}}\right),
	\end{equation*}
	where  ${\mathcal N}$ is a standard normal variable. By substituting the values of $a$ 
    and $T_n$, we obtain
	\begin{eqnarray*}
		\prob(A_n) &\sim& \frac{1}{\sqrt{2\pi}} \frac {(a/\ln n)^{1/2}}{r} \exp\left( - \frac{r^2}{2a} \, \ln n     \right)
		\\
		&=&  \frac{1}{2\sqrt{\pi}}  (\ln n)^{-1/2}    n^{-1}. 
	\end{eqnarray*}
	The series $\sum_n \prob(A_n)$ diverges, thus by Borel--Cantelli lemma the events $A_n$ occur infinitely often.
	
	Let now $T>0$ be such that
	\begin{equation*}
	     W(T)- W\left(T-\tfrac{a}{\ln T}\right) \ge r.
	\end{equation*}
	Then for every function $g$ such that
	\begin{equation}\label{gjump}
		\big| g\left(T-\tfrac{a}{\ln T}\right) -  W\left(T-\tfrac{a}{\ln T}\right) \big| \le \frac r2;
		\qquad  g(T)=W(T),
	\end{equation}
	it is true that
	\begin{equation*}
	   g(T)- g\left(T-\tfrac{a}{\ln T}\right) 
          \ge W(T) -\left(  W\left(T-\tfrac{a}{\ln T}\right) +\frac r2   \right)
	   \ge r-\frac r2 = \frac r2.
	\end{equation*}
	If $g\in \text{AC}[0,T]$, then by convexity of the function $\varphi$ Jensen inequality yields
	\[
		\int_{T-\tfrac{a}{\ln T}}^{T} \varphi(g'(t))\, dt 
		\ge  \frac{a}{\ln T}\  \varphi\left( \frac{ g(T)- g\left(T-\tfrac{a}{\ln T}\right)}{a/\ln T}\right).
    \]
    For sufficiently large $T$ we infer
	\[
        \int_{T-\tfrac{a}{\ln T}}^{T} \varphi(g'(t))\, dt 
		\ge \frac{a}{\ln T} \, 
        \exp\left( \frac{\lambda r \cdot r/2}{a/\ln T}  \right) = \frac{a}{\ln T}\, T^{\lambda}.
	\]
	Since the taut string $\eta_{T,r}$ satisfies conditions \eqref{gjump}, it also obeys
 inequalities
	\begin{equation*}
	   \int_0^T \varphi(\eta'_{T,r}(t)) dt \ge  \int_{T-\tfrac{a}{\ln T}}^T \varphi(\eta'_{T,r}(t))\, dt 
	\ge  \frac{a}{\ln T}\, T^{\lambda}\gg T,
	\end{equation*}
	where in the last step we used the assumption $\lambda>1$. Therefore, since the events $A_n$ occur infinitely often, the stated estimates hold for arbitrary $r>0$ and some arbitrarily large $T$, which proves the theorem.

%%%%%%%%%%%%%%%%%%%%%%%%%%%%%%%%%%%%%%%%%%%%%%%%%%%%%%%%%%%%
\subsection{The limit computation for quadratic function \texorpdfstring{$\vartheta_2$}{vartheta2}}
%%%%%%%%%%%%%%%%%%%%%%%%%%%%%%%%%%%%%%%%%%%%%%%%%%%%%%%%%%%%%
\label{sec:t2_calc}

As we know from 
\eqref{main_asconv},
\begin{equation*}
	R_r(\vartheta_2)
	\aseq
	\int\limits_{\mathbb R}	\vartheta_2(u)p_{\infty,r}(u)du
	=	\int\limits_{-\infty}^{\infty} u^2r \frac{(ru\coth(ru)-1)}{\sh^2(ru)} \, du
	:=	I.
\end{equation*}
Let us show that $I=\frac{\pi^2}{6r^2}$. Indeed,
\begin{equation*}
	I
	=	r^{-2} \int\limits_{\mathbb R} v^2 \frac{(v\coth(v)-1)}{\sh^2(v)} \, dv
	=	r^{-2} \bigg(
	  \int\limits_{\mathbb R} \frac{v^3\ch(v)}{\sh^3(v)}\, dv
	  -
	  \int\limits_{\mathbb R} \frac{v^2}{\sh^2(v)} \, dv
	         \bigg).
\end{equation*}
Integration by parts yields
\begin{equation*}
	\int\limits_{\mathbb R} \frac{v^3\ch(v)}{\sh^3(v)} \, dv
	= \int\limits_{\mathbb R}	v^3 \,
  	d\bigg(
	     \frac{1}{2\sh^2(v)}
	   \bigg)
	= \frac{3}{2} \int\limits_{\mathbb R} \frac{v^2}{\sh^2(v)} \, dv.
\end{equation*}
Therefore,
\begin{equation*}
	I=	\frac{r^{-2}}{2}\int\limits_{\mathbb R} \frac{v^2}{\sh^2(v)} \, dv
	=r^{-2} \int\limits_{0}^{+\infty} \frac{v^2}{\sh^2(v)} \, dv
	=\frac{\pi^2}{6r^2}.
\end{equation*}
where the last equality is a special case of~\cite[formula (3.527.2)]{GR_2015_Eng}.

%%%%%%%%%%%%%%%%%%%%%%%%%%%%%%%%%%%%%%%%%%%%%%%%
\subsection{Truncated variations and absolutely continuous functions}
%%%%%%%%%%%%%%%%%%%%%%%%%%%%%%%%%%%%%%%%%%%%%%%%
\label{subsec:TVAC}

\begin{proof}[Proof of Lemma~{\normalfont\ref{lemma:TVr_AC}}] 
	Let $\varepsilon>0$, and a function $h$ on $[a,b]$ attains the value of 
    $TV^{r-\varepsilon}(f,[a,b])$,
	i.e. $TV^0(h,[a,b])= TV^{r-\varepsilon}(f,[a,b])$ and  
    $\|h-f\|_{[a,b]}\le \tfrac{r-\varepsilon}{2}$.
	
	Let us extend $h$ by the constants from $[a,b]$ to $\mathbb{R}$ and approximate $h$ by convolutions with smooth kernels
	\begin{equation*}
	h^\delta(t):= \int_\mathbb{R} \Phi^\delta(v) h(t+v) dv 
               = \int_\mathbb{R} \Phi^\delta(s-t) h(s) ds.
	\end{equation*}
	We assume the usual properties of the kernels:
	\begin{eqnarray*}
		\Phi^\delta(\cdot)\ge 0; \quad \int_\mathbb{R} \Phi^\delta(v) dv =1 ; \quad \int_\mathbb{R} |(\Phi^\delta)'(v)| dv<\infty;
		\\
		\lim_{\delta\to 0}  \int_{|v|>\theta} \Phi^\delta(v) dv =0 
        \qquad \textrm{for every } \theta>0.
	\end{eqnarray*}
	Let us prove that
	
	1) $h^\delta\in \text{AC}[a,b]$;
	
	2) $TV^0(h^\delta,[a,b]) \le  TV^0(h,[a,b])= TV^{r-\varepsilon}(f,[a,b])$;
	
	3) For sufficiently small $\delta$ it is true that $\|h^\delta-f\|_{[a,b]}\le \tfrac{r}{2}$.
	
	Then we have
	\begin{equation*}
	\begin{aligned}
	\inf\left\{ \int_a^b |g'(t)| dt \ ; \  g\in \text{AC}[a,b], \|g-f\|_{[a,b]}\le \tfrac r2  \right\}
	&\leq
	TV^0(h^\delta,[a,b]) \\
	&\leq TV^{r-\varepsilon}(f,[a,b]).
	\end{aligned}
	\end{equation*}
	Taking into account that the function $r\mapsto TV^r(f,[a,b])$ is continuous (cf.~\cite{Lochowski2013}), by letting $\varepsilon$ to zero,
	we find
	\begin{equation*}
	  \inf\left\{ \int_a^b |g'(t)| dt \ ; \  g\in \text{AC}[a,b], \|g-f\|_{[a,b]}\le 
        \tfrac r2  \right\}
	  \le   TV^{r}(f,[a,b]).
	\end{equation*}
	The opposite inequality is obvious.
	
	It remains to prove properties 1) \,--\, 3).
	
	1) The function $h^\delta$ is differentiable and its derivative is uniformly bounded because
	\begin{eqnarray*}
		|(h^\delta)'(t)| = \left| \int_\mathbb{R} (\Phi^\delta)'(s-t) h(s) ds  \right|
		&\le&
		\int_\mathbb{R} \left|(\Phi^\delta)'(s-t)\right| ds  \max_{s\in[a,b]} |h(s)|
		\\
		&=&  \int_\mathbb{R} \left|(\Phi^\delta)'(v)\right| dv  \max_{s\in[a,b]} |h(s)|.
	\end{eqnarray*}
	Therefore, $h^\delta\in \text{Lip}[a,b]\subset \text{AC}[a,b]$.
	
	2) For all $a_1\le...\le a_m$ from $[a,b]$ we have
	\begin{eqnarray*}
		\sum_{i=1}^{m-1} |h^\delta(a_{i+1})- h^\delta(a_{i})|
		&=&   \sum_{i=1}^{m-1} \left| \int_\mathbb{R} \Phi^\delta(v) \left(h(a_{i+1}+v)- h(a_{i}+v)\right)dv \right|
		\\
		&\le&  \int_\mathbb{R} \Phi^\delta(v) \sum_{i=1}^{m-1} \left| h(a_{i+1}+v)- h(a_{i}+v)\right| \, dv
		\\
		&\le& \int_\mathbb{R} \Phi^\delta(v)\, dv \ TV^0(h,[a,b]) = TV^0(h,[a,b]).
	\end{eqnarray*}
	Therefore, $TV^0(h^\delta,[a,b])\le TV^0(h,[a,b])$.
	
	3) We have to prove that for small $\delta$ and all $\tau\in[a,b]$ it is true that
	\begin{equation*}
    	f(\tau) - \frac{r}2 \le h^\delta(\tau)\le f(\tau)+ \frac{r}{2}.
	\end{equation*}
	We will prove only the first inequality; the second one can be proved in the same way. For every  $\tau$ it is true that
	\begin{eqnarray*}
		h^\delta(\tau) &=& \int_\mathbb{R} \Phi^\delta(v) h(\tau+v)\, dv
		\ge \int_\mathbb{R} \Phi^\delta(v) \left[ f(\tau+v)- \frac{r-\varepsilon}{2}\right]  dv
		\\
		&=&   \int_\mathbb{R} \Phi^\delta(v) f(\tau+v)\, dv - \frac{r-\varepsilon}{2}.
	\end{eqnarray*}
	Therefore,
	\begin{equation*}
	   h^\delta(\tau)-f(\tau) 
         \ge  \int_\mathbb{R} \Phi^\delta(v) [f(\tau+v)-f(\tau)] \, dv - 
         \frac{r-\varepsilon}{2}.
	\end{equation*}
	Moreover, for every $\theta>0$
	\begin{eqnarray*}
		&& \left| \int_\mathbb{R} \Phi^\delta(v) [f(\tau+v)-f(\tau)] dv \right|
		\le  \int_\mathbb{R} \Phi^\delta(v) \left|f(\tau+v)-f(\tau)\right| dv
		\\
		&\le& \max_{|v|\le\theta}|f(\tau+v)-f(\tau)| + 2\max_{[a,b]} |f|  \int_{|v|>\theta} \Phi^\delta(v) dv.
	\end{eqnarray*}
	
	Here, by a kernels' property, the second term converges to zero for every $\theta$. The first term can be made arbitrarily small uniformly in $\tau$ by the choice of $\theta$ (it is here when the continuity of $f$ is used).
	
	Therefore, for sufficiently small $\delta$ and all $\tau\in[a,b]$ it is true that
	\begin{equation*}
	  \left| \int_\mathbb{R} \Phi^\delta(v) [f(\tau+v)-f(\tau)]\, dv \right| 
        \le \frac{\varepsilon}{2},
	\end{equation*}
	which yields the desired
	\begin{equation*}
	    h^\delta(\tau)-f(\tau) \ge - \frac {r}{2}.
	\end{equation*}
\end{proof}
\medskip

The representation \eqref{eq:UTVAC} can be proved exactly in the same way, just with replacement of the function $x\mapsto |x|$ with the function $x\mapsto x^+$.
\medskip

Finally, let us prove representation \eqref{eq:UTV_integral}. Let $g\in \text{AC}[a,b]$
and $a\leq t_1<\dots<t_n\leq b$. Then
\[
   \sum_{k=1}^{n-1} (g(t_{k+1})-g(t_k))^+ 
   \le  \sum_{k=1}^{n-1} \int_{t_{k}}^{t_{k+1}} g'(t)^+\, dt \le \int_a^b g'(t)^+\, dt.
\]
By maximizing over partitions, we obtaion
\[
      UTV^0(g) \le \int_a^b g'(t)^+\, dt.
\]

In order to prove the opposite bound, let us fix an $\varepsilon>0$ and choose a step function $h$ such that $||g'-h||_1:=||g'-h||_{L_1[a,b]}\le \varepsilon$. 
Let $a= t_1<\dots<t_n=b$ be a partition such that the function $h$ is constant on all corresponding intervals.

If an interval $[t_{k},t_{k+1}]$  is such that $h\big|_{[t_{k},t_{k+1}]}\ge 0$,
then
\[
   \int_{t_{k}}^{t_{k+1}} g'(t)^- dt \leq \int_{t_{k}}^{t_{k+1}} |h(t)-g'(t)| \, dt,
\]
hence,
\begin{eqnarray*}
     (g(t_{k+1})-g(t_k))^+  &\ge&  g(t_{k+1})-g(t_k) 
     =  \int_{t_{k}}^{t_{k+1}} g'(t)\, dt
\\
    &=& \int_{t_{k}}^{t_{k+1}} g'(t)^+\, dt -  \int_{t_{k}}^{t_{k+1}} g'(t)^-\, dt
\\
   &\geq& \int_{t_{k}}^{t_{k+1}} g'(t)^+\, dt - \int_{t_{k}}^{t_{k+1}} |h(t)-g'(t)| \, dt.
\end{eqnarray*}

Otherwise, if an interval $[t_{k},t_{k+1}]$ is such that $h\big|_{[t_{k},t_{k+1}]}< 0$,
then
\[
   \int_{t_{k}}^{t_{k+1}} g'(t)^+ dt \leq \int_{t_{k}}^{t_{k+1}} |h(t)-g'(t)| \, dt.
\]
By summing up over all over all intervals and splitting the sum in two corresponding
groups of terms we obtain
\begin{eqnarray*}
  UTV^0(g) &\geq& \sum_{k=1}^{n-1} (g(t_{k+1})-g(t_k))^+  
  \geq \sum_{k:h\big|_{[t_{k},t_{k+1}]}\ge 0}  (g(t_{k+1})-g(t_k))^+  
\\
     &\geq& \sum_{k:h\big|_{[t_{k},t_{k+1}]}\ge 0} 
      \left[ \int_{t_{k}}^{t_{k+1}} g'(t)^+\, dt - \int_{t_{k}}^{t_{k+1}} |h(t)-g'(t)| \, dt\right]
\\
       &\geq& \left[ \sum_{k:h\big|_{[t_{k},t_{k+1}]}\ge 0} 
       \int_{t_{k}}^{t_{k+1}} g'(t)^+\, dt \right] - ||h-g'||_1
\\
       &\geq&  \int_a^b g'(t)^+\,dt - \left[ \sum_{k:h\big|_{[t_{k},t_{k+1}]}< 0} 
       \int_{t_{k}}^{t_{k+1}} g'(t)^+\, dt \right] - ||h-g'||_1
\\
       &\geq&  \int_a^b g'(t)^+\,dt - \left[ \sum_{k:h\big|_{[t_{k},t_{k+1}]}< 0} 
       \int_{t_{k}}^{t_{k+1}}  |h(t)-g'(t)| \, dt \right] - ||h-g'||_1
\\
       &\geq&  \int_a^b g'(t)^+\,dt - 2||h-g'||_1 \ge \int_a^b g'(t)^+\,dt - 2\varepsilon.
\end{eqnarray*}
By letting $\varepsilon\to 0$, we obtain 
\[
   UTV^0(g) \geq \int_a^b g'(t)^+\,dt.
\]

%%%% END OF OUR ARTICLE  %%%%%%%%%%%%%%%%%%%%%%%%%%%%%%%%%%%%%%%%%%%

%%%%%%%%%%%%%%%%%%%%%%%%%%%%%%%%%%%%%%%%%%%%%%%%%%%%%%%%%%%%%%%%%%%
%%                                                               %%
%% Supplementary Material, if any, should be provided in         %%
%% {supplement} environment  with title and short description.   %%
%%                                                               %%
%%%%%%%%%%%%%%%%%%%%%%%%%%%%%%%%%%%%%%%%%%%%%%%%%%%%%%%%%%%%%%%%%%%

%%\begin{supplement}
%%\stitle{Title of Supplement A.}
%%\sdescription{Short description of Supplement A.}
%%\end{supplement}
%%\begin{supplement}
%%\stitle{Title of Supplement B.}
%%\sdescription{Short description of Supplement B.}
%%\end{supplement}

%%%%%%%%%%%%%%%%%%%%%%%%%%%%%%%%%%%%%%%%%%%%%%%%%%%%%%%%%%%%%%%%%%%
%%                                                               %%
%% Use the two commands below for producing your bibliography    %%
%% with bibtex, then comment again the commands and include the  %%
%% content of the .bbl file in this file below the commands.     %%
%%                                                               %%
%%%%%%%%%%%%%%%%%%%%%%%%%%%%%%%%%%%%%%%%%%%%%%%%%%%%%%%%%%%%%%%%%%%

\bibliographystyle{ugost2008mod.bst}
\bibliography{WienerEnergy.bib}

@Article{Lifshits2015,
  author    = {Lifshits, M. and Setterqvist, E.},
  journal   = {Stochastic Processes and their Applications},
  title     = {Energy of taut strings accompanying Wiener process},
  year      = {2015},
  number    = {2},
  pages     = {401--427},
  volume    = {125},
}

@Article{Schertzer2018,
  author    = {Schertzer, E.},
  journal   = {Stochastic Processes and their Applications},
  title     = {Renewal structure of the Brownian taut string},
  year      = {2018},
  number    = {2},
  pages     = {487--504},
  volume    = {128},
}

@Article{Lochowski2008,
  author    = {Łochowski, R.},
  journal   = {Bulletin of the Polish Academy of Sciences. Mathematics},
  title     = {On truncated variation of Brownian motion with drift},
  year      = {2008},
  number    = {3-4},
  pages     = {267--281},
  volume    = {56},
}

@Article{Lochowski2011,
  author    = {Łochowski, R.},
  journal   = {LMS Journal of Computation and Mathematics},
  title     = {On the double Laplace transform of the truncated variation of a Brownian motion with drift},
  year      = {2011},
  number    = {1},
  pages     = {281--292},
  volume    = {19},
}

@Article{Lochowski2013,
  author    = {Łochowski, R. and Miłoś, P.},
  journal   = {Stochastic Processes and their Applications},
  title     = {On truncated variation, upward truncated variation and downward truncated variation for diffusions},
  year      = {2013},
  number    = {2},
  pages     = {446--474},
  volume    = {123},
}

@Article{Lochowski2013b,
  author    = {Łochowski, R.},
  journal   = {Colloquium Mathematicum},
  title     = {On a generalization of the Hahn-Jordan decomposition for real c\'adl\'ag functions},
  year      = {2013},
  number    = {1},
  pages     = {121--137},
  volume    = {132},
}

@Article{Blinova2020,
  author  = {Blinova, D.I. and Lifshits, M.A.},
  journal = {Journal of Mathematical Sciences N.Y.},
  title   = {Energy of taut strings accompanying  Wiener  process and random walk
in a band of variable width},
  year    = {2022},
  pages   = {573--588},
  volume  = {268},
  number={5},
}

@Article{Nikitin2024,
  author  = {Lifshits, M.A. and Nikitin, S.E.},
  journal = {Theory of Probability and its Applications},
  title   = {Energy saving approximation of Wiener process under unilateral constraints},
  year    = {2024},
  pages   = {59--70},
  volume  = {69},
  number =  {1}
}

@Article{Siuniaev2021,
  author  = {Lifshits, M.A. and Siuniaev, A.A.},
  journal = {Probability and Mathematical Statistics},
  title   = {Energy of taut strings accompanying random walk},
  year    = {2021},
  pages   = {9--23},
  volume  = {41},
  number =  {1}
}

@Book{GR_2015_Eng,
  author    = {Gradshteyn, I. and Ryzhik, I.},
  publisher = {Academic Press},
  title     = {Table of integrals, series, and products},
  year      = {2015},
  address   = {Waltham, MA},
  edition   = {Eighth},
  isbn      = {0123849349},
  pagetotal = {1133},
  ppn_gvk   = {1657078310},
}

@Article{Grasmair2007,
  author  = {Grasmair, M.},
  journal = {Journal of Mathematical Imaging and Vision},
  title   = {The equivalence of the taut string algorithm and {BV}-regularization},
  year    = {2007},
  pages   = {59--66},
  volume  = {27}
}

@Book{Daley2005,
  author    = {Daley, D. J. and Vere-Jones, D.},
  title     = {An Introduction to The Theory of Point Processes: Volume I: Elementary Theory and Methods},
  year      = {2005},
  address   = {New York NY},
  edition   = {2. corr. print.},
  pagetotal = {469},
  ppn_gvk   = {1187840262},
  publisher = {Springer},
}

% add below the content of your .bbl file produced by bibtex.
{\bf Mikhail Lifshits},   mikhail@lifshits.org.

{\bf Andrei Podchishchailov}, andrei.podch@yandex.ru.
\medskip

\noindent St.Petersburg State University.
Russia, 191023, St.Petersburg, Universitetskaya emb. 7/9.

%%%%%%%%%%%%%%%%%%%%%%%%%%%%%%%%%%%%%%%%%%%%%%%%%%%%%%%%%%%%%%%%%%%
%%                                                               %%
%% You may add acknowledgments (optional).                       %%
%%                                                               %%
%%%%%%%%%%%%%%%%%%%%%%%%%%%%%%%%%%%%%%%%%%%%%%%%%%%%%%%%%%%%%%%%%%%
\begin{acks}
We are very grateful to the referee for careful reading of our article and for
useful advice pointing out the unclear passages and helping to improve the presentation.
\end{acks}

%%%%%%%%%%%%%%%%%%%%%%%%%%%%%%%%%%%%%%%%%%%%%%%%%%%%%%%%%%%%%%%%%%%
%%                                                               %%
%% You have reached the end of your document.                    %%
%%                                                               %%
%%%%%%%%%%%%%%%%%%%%%%%%%%%%%%%%%%%%%%%%%%%%%%%%%%%%%%%%%%%%%%%%%%%

\end{document}